\documentclass[a4paper]{amsart}

\usepackage[utf8]{inputenc}
\usepackage[english]{babel}

\usepackage[T1]{fontenc}
\usepackage{lmodern}  

\usepackage{my-math-symbol}
\usepackage{my-cleveref}
\usepackage{my-equation-numbering}
\usepackage{comment}
\usepackage{tikz}
\usepackage{color}
\usetikzlibrary{cd}
\usepackage[colorinlistoftodos]{todonotes}

\usepackage[non-sorted-cites,initials,alphabetic,nobysame]{amsrefs}

\AtBeginDocument{%
	\def\MR#1{}
}

\title{CR $Q$-curvature and CR pluriharmonic functions}
\author{Yuya Takeuchi}
\address{Division of Mathematics \\ Faculty of Pure and Applied Sciences \\ University of Tsukuba
	\\ 1-1-1 Tennodai, Tsukuba, Ibaraki 305-8571 Japan}
\email{ytakeuchi@math.tsukuba.ac.jp, yuya.takeuchi.math@gmail.com}



\thanks{This work was supported by JSPS Research Fellowship for Young Scientists
and JSPS KAKENHI Grant Numbers JP19J00063 and JP21K13792.}

\begin{document}

\begin{abstract}
	In this paper,
	we show that the CR $Q$-curvature is orthogonal to
	the space of CR pluriharmonic functions on any closed strictly pseudoconvex CR manifold
	of dimension at least five.
	To this end,
	we obtain a cohomological expression of the integral
	of the product of the CR $Q$-curvature and a CR pluriharmonic function.
\end{abstract}

\maketitle

\section{Introduction}
\label{section:introduction}

In conformal geometry,
the $Q$-curvature $Q_{g}$ introduced by Branson~\cite{Branson1995} has been of great importance.
It is a natural Riemannian invariant on a $2m$-dimensional manifold
and transforms as follows under the conformal change $\whg = e^{2 \Upsilon} g$:
\begin{equation}
	e^{2m \Upsilon} Q_{\whg}
	= Q_{g} + P_{g} \Upsilon,
\end{equation}
where $P_{g}$ is the critical GJMS operator~\cite{Graham-Jenne-Mason-Sparling1992}.
Since $P_{g}$ is self-adjoint and annihilates constant functions,
the integral of $Q_{g}$ defines a non-trivial global conformal invariant
of an even-dimensional closed conformal manifold.

In CR geometry,
the CR $Q$-curvature $Q_{\theta}$ has been introduced by Fefferman and Hirachi~\cite{Fefferman-Hirachi2003}
using a conformal manifold associated with a non-degenerate CR manifold.
It is a natural pseudo-Hermitian invariant on a $(2n+1)$-dimensional non-degenerate CR manifold
and transforms as follows under the conformal change $\whxth = e^{\Upsilon} \theta$:
\begin{equation}
	e^{(n+1) \Upsilon} Q_{\whxth}
	= Q_{\theta} + P_{\theta} \Upsilon,
\end{equation}
where $P_{\theta}$ is the critical CR GJMS operator~\cite{Gover-Graham2005}.
Since $P_{\theta}$ is formally self-adjoint and annihilates constant functions,
the integral of $Q_{\theta}$ defines a global CR invariant
of a closed non-degenerate CR manifold similar to conformal geometry.
However,
it turns out that this is identically zero on strictly pseudoconvex CR manifolds~\cite{Marugame2018}.
Moreover,
the CR $Q$-curvature itself is identically zero for pseudo-Einstein contact forms~\cite{Fefferman-Hirachi2003},
which are contact forms satisfying a (weak) Einstein condition.
These facts motivate us to pose the following problem:

\begin{problem}
\label{problem:contact-form-with-zero-CR-Q-curvature}
	Does any closed strictly pseudoconvex CR manifold
	admit a contact form with zero CR $Q$-curvature?
\end{problem}

This problem has been solved affirmatively
for three-dimensional embeddable CR manifolds~\cite{Takeuchi2020-Paneitz}.
However,
it is open in general.

There exists an obstruction to the existence of a contact form with zero CR $Q$-curvature.
Let $(M, T^{1, 0} M, \theta)$ be a closed pseudo-Hermitian manifold of dimension $2 n + 1$,
and $f \in \ker P_{\theta}$.
Then the integral
\begin{equation}
	\calQ(f)
	= \int_{M} f Q_{\theta} \theta \wedge (d \theta)^{n}
\end{equation}
is independent of the choice of $\theta$
and defines a CR invariant functional
\begin{equation}
	\calQ \colon \ker P_{\theta} \to \mathbb{R}.
\end{equation}
The author~\cite{Takeuchi2020-preprint} has shown that
in the case that $(M, T^{1, 0} M)$ is embeddable,
it admits a contact form with zero CR $Q$-curvature
if and only if the functional $\calQ$ is identically zero.

The aim of this paper is to show $\calQ$ is identically zero
on the space $\scrP$ of CR pluriharmonic functions.

\begin{theorem}
\label{thm:orthogonality-of-CR-Q-curvature}
	Let $(M, T^{1, 0} M)$ be a closed strictly pseudoconvex CR manifold of dimension $2n+1 \geq 5$..
	Then $\calQ$ is identically zero on $\scrP$;
	that is,
	$Q_{\theta}$ is orthogonal to $\scrP$ for any contact form $\theta$.
\end{theorem}

To this end,
we show a cohomological expression of the integral
of the product of the CR $Q$-curvature and a CR pluriharmonic function,
which is of independent interest.

\begin{theorem}
\label{thm:cohomological-expression-of-CR-Q-curvature}
	Let $(M, T^{1, 0} M)$ be as in \cref{thm:orthogonality-of-CR-Q-curvature}.
	Then for any $\Upsilon \in \scrP$ and any contact form $\theta$,
	\begin{equation}
	\label{eq:cohomological-expression-of-CR-Q-curvature}
		\langle [d^{c}_{\CR} \Upsilon] \cup c_{1}(T^{1, 0} M)^{n}, [M] \rangle
		= \frac{(n+2)^{n}}{n(n!)^{2} (2 \pi)^{n}} \int_{M} \Upsilon Q_{\theta} \theta \wedge (d \theta)^{n},
	\end{equation}
	where $d^{c}_{\CR}$ is defined by \cref{eq:definition-of-d^c_CR}.
\end{theorem}

Note that \cref{eq:cohomological-expression-of-CR-Q-curvature} for the three-dimensional case
has been proved by the author~\cite{Takeuchi2020-Paneitz}.
Moreover,
Case~\cite{Case2021-Q-preprint} has proved that any closed non-degenerate CR five-manifold
satisfies \cref{eq:cohomological-expression-of-CR-Q-curvature}
by using the bigraded Rumin complex via differential forms~\cite{Case2021-Rumin-preprint}.
It would be an interesting problem whether \cref{eq:cohomological-expression-of-CR-Q-curvature} holds
on all closed non-degenerate CR manifolds.

This paper is organized as follows.
In \cref{section:CR-geometry} (resp.\ \cref{section:strictly-pseudoconvex-domains}),
we recall basic facts on CR manifolds (resp.\ strictly pseudoconvex domains).
\cref{section:proof-of-main-theorems} is devoted to proofs of \cref{thm:orthogonality-of-CR-Q-curvature,thm:cohomological-expression-of-CR-Q-curvature}.

\medskip

\noindent
\emph{Notation.}
We use Einstein's summation convention and assume that 
lowercase Greek indices $\alpha, \beta, \gamma, \dots$ run from $1, \dots, n$.

Suppose that a function $I(\epsilon)$ admits an asymptotic expansion,
as $\epsilon \to + 0$,
\begin{equation}
	I(\epsilon)
	=\sum_{m = 1}^{k} a_{m} \epsilon^{-m} + b \log \epsilon + O(1).
\end{equation}
Then the \emph{logarithmic part} $\lp I(\epsilon)$ of $I(\epsilon)$ is the constant $b$.

\medskip

\section{CR geometry}
\label{section:CR-geometry}

\subsection{CR structures}
\label{subsection:CR-structures}

Let $M$ be a smooth $(2n+1)$-dimensional manifold without boundary.
A \emph{CR structure} is a rank $n$ complex subbundle $T^{1, 0} M$
of the complexified tangent bundle $T M \otimes \mathbb{C}$ such that
\begin{equation}
	T^{1, 0}M \cap T^{0, 1} M = 0, \qquad
	[\Gamma(T^{1, 0} M), \Gamma(T^{1, 0} M)] \subset \Gamma(T^{1, 0} M),
\end{equation}
where $T^{0, 1} M$ is the complex conjugate of $T^{1, 0} M$ in $T M \otimes \mathbb{C}$.
A typical example of CR manifolds is a real hypersurface $M$ in an $(n+1)$-dimensional complex manifold $X$;
this $M$ has the canonical CR structure
\begin{equation}
	T^{1, 0} M
	= T^{1, 0} X |_{M} \cap (T M \otimes \mathbb{C}).
\end{equation}
Introduce an operator $\delb_{b} \colon C^{\infty}(M) \to \Gamma((T^{0, 1} M)^{*})$ by
\begin{equation}
	\delb_{b} f = (d f)|_{T^{0, 1} M}.
\end{equation}
A smooth function $f$ is called a \emph{CR holomorphic function}
if $\delb_{b} f = 0$.
A \emph{CR pluriharmonic function} is a real-valued smooth function
that is locally the real part of a CR holomorphic function.
We denote by $\scrP$ the space of CR pluriharmonic functions.

A CR structure $T^{1, 0} M$ is said to be \emph{strictly pseudoconvex}
if there exists a nowhere-vanishing real one-form $\theta$ on $M$
such that
$\theta$ annihilates $T^{1, 0} M$ and
\begin{equation}
	- \sqrt{-1} d \theta (Z, \overline{Z}) > 0, \qquad
	0 \neq Z \in T^{1, 0} M.
\end{equation}
We call such a one-form a \emph{contact form}.
The triple $(M, T^{1, 0} M, \theta)$ is called a \emph{pseudo-Hermitian manifold}.
Denote by $T$ the \emph{Reeb vector field} with respect to $\theta$; 
that is,
the unique vector field satisfying
\begin{equation}
	\theta(T) = 1, \qquad d \theta(T, \cdot) = 0.
\end{equation}
Let $(\tensor{Z}{_{\alpha}})$ be a local frame of $T^{1, 0} M$,
and set $\tensor{Z}{_{\ovxa}} = \overline{\tensor{Z}{_{\alpha}}}$.
Then
$(T, \tensor{Z}{_{\alpha}}, \tensor{Z}{_{\ovxa}})$ gives a local frame of $T M \otimes \mathbb{C}$,
called an \emph{admissible frame}.
Its dual frame $(\theta, \tensor{\theta}{^{\alpha}}, \tensor{\theta}{^{\ovxa}})$
is called an \emph{admissible coframe}.
The two-form $d \theta$ is written as
\begin{equation}
	d \theta
	= \sqrt{-1} \tensor{l}{_{\alpha}_{\ovxb}} \tensor{\theta}{^{\alpha}} \wedge \tensor{\theta}{^{\ovxb}},
\end{equation}
where $(\tensor{l}{_{\alpha}_{\ovxb}})$ is a positive definite Hermitian matrix.
We use $\tensor{l}{_{\alpha}_{\ovxb}}$ and its inverse $\tensor{l}{^{\alpha} ^{\ovxb}}$
to raise and lower indices of tensors.

\subsection{Tanaka-Webster connection}
\label{subsection:TW-connection}

A contact form $\theta$ induces a canonical connection $\nabla$,
called the \emph{Tanaka-Webster connection} with respect to $\theta$.
It is defined by
\begin{equation}
	\nabla T
	= 0,
	\quad
	\nabla \tensor{Z}{_{\alpha}}
	= \tensor{\omega}{_{\alpha}^{\beta}} \tensor{Z}{_{\beta}},
	\quad
	\nabla \tensor{Z}{_{\ovxa}}
	= \tensor{\omega}{_{\ovxa}^{\ovxb}} \tensor{Z}{_{\ovxb}}
	\quad
	\pqty{ \tensor{\omega}{_{\ovxa}^{\ovxb}}
	= \overline{\tensor{\omega}{_{\alpha}^{\beta}}} }
\end{equation}
with the following structure equations:
\begin{gather}
\label{eq:str-eq-of-TW-conn1}
	d \tensor{\theta}{^{\beta}}
	= \tensor{\theta}{^{\alpha}} \wedge \tensor{\omega}{_{\alpha}^{\beta}}
	+ \tensor{A}{^{\beta}_{\ovxa}} \theta \wedge \tensor{\theta}{^{\ovxa}}, \\
\label{eq:str-eq-of-TW-conn2}
	d \tensor{l}{_{\alpha}_{\ovxb}}
	= \tensor{\omega}{_{\alpha}^{\gamma}} \tensor{l}{_{\gamma}_{\ovxb}}
	+ \tensor{l}{_{\alpha}_{\ovxg}} \tensor{\omega}{_{\ovxb}^{\ovxg}}.
\end{gather}
The tensor $\tensor{A}{_{\alpha}_{\beta}} = \overline{\tensor{A}{_{\ovxa}_{\ovxb}}}$
is shown to be symmetric and is called the \emph{Tanaka-Webster torsion}.
We denote the components of a successive covariant derivative of a tensor
by subscripts preceded by a comma,
for example, $\tensor{K}{_{\alpha}_{\ovxb}_{,}_{\gamma}}$;
we omit the comma if the derivatives are applied to a function.
Consider commutators of covariant derivatives.
The commutators of the second derivatives for $u \in C^{\infty}(M)$ are given by
\begin{equation}
\label{eq:commutator-of-covariant-derivatives}
	2 \tensor{u}{_{[}_{\alpha}_{\beta}_{]}}
	= 0,
	\quad
	2 \tensor{u}{_{[}_{\alpha}_{\ovxb}_{]}}
	= \sqrt{-1} \tensor{l}{_{\alpha}_{\ovxb}} u_{0},
	\quad
	2 \tensor{u}{_{[}_{0}_{\alpha}_{]}}
	= \tensor{A}{_{\alpha}_{\beta}} \tensor{u}{^{\beta}},
\end{equation}
where $[\dotsm]$ means the anti-symmetrization over the enclosed indices.
Define the \emph{sub-Laplacian} $\Delta_{b}$ by
\begin{equation}
\label{eq:definition-of-sub-Laplacian}
	\Delta_{b} u
	= - \tensor{u}{_{\ovxb} ^{\ovxb}}
		- \tensor{u}{_{\alpha}^ {\alpha}}
\end{equation}
for $u \in C^{\infty}(M)$.
It follows from \cref{eq:commutator-of-covariant-derivatives} that
\begin{equation}
\label{eq:another-form-of-sub-Laplacian}
	\Delta_{b} u
	= -2 \tensor{u}{_{\ovxb}^{\ovxb}} - \sqrt{-1} n \tensor{u}{_{0}}
	= -2 \tensor{u}{_{\alpha}^{\alpha}} + \sqrt{-1} n \tensor{u}{_{0}}.
\end{equation}

Let $\hat{\theta} = e^{\Upsilon} \theta$ be another contact form,
and denote by $\widehat{T}$ the corresponding Reeb vector field.
The admissible coframe corresponding to $(\widehat{T}, \tensor{Z}{_{\alpha}}, \tensor{Z}{_{\ovxb}})$
is given by
\begin{equation}
\label{eq:transformation-law-of-admissible-coframe}
	(\hat{\theta},
	\tensor{\hat{\theta}}{^{\alpha}} = \tensor{\theta}{^{\alpha}} + \sqrt{-1} \tensor{\Upsilon}{^{\alpha}} \theta,
	\tensor{\hat{\theta}}{^{\ovxb}} = \tensor{\theta}{^{\ovxb}} - \sqrt{-1} \tensor{\Upsilon}{^{\ovxb}} \theta).
\end{equation}
Under this conformal change,
we have
\begin{equation}
\label{eq:transformation-law-of-sub-Laplacian}
	e^{\Upsilon} \widehat{\Delta}_{b} u
	= \Delta_{b} u - n \tensor{\Upsilon}{^{\alpha}} \tensor{u}{_{\alpha}} - n \tensor{\Upsilon}{^{\ovxb}} \tensor{u}{_{\ovxb}};
\end{equation}
see \cite{Stanton1989}*{Lemma 1.8} for example.

\subsection{CR pluriharmonic functions}
\label{subsection:CR-pluriharmonic-functions}

Let $(M, T^{1, 0} M, \theta)$ be a pseudo-Hermitian manifold of dimension $2n+1$.
We first introduce a CR analogue of $d^{c} = (\sqrt{-1} / 2) (\delb - \del)$.
Such an operator has been defined by Case~\cite{Case2021-Rumin-preprint} via the bigraded Rumin complex.
However,
we give an elementary construction following~\cite{Takeuchi2020-Paneitz}
for the reader's convenience.

\begin{lemma}
	The differential operator
	\begin{equation}
	\label{eq:definition-of-d^c_CR}
		d^{c}_{\CR} \colon C^{\infty}(M) \to \Omega^{1} (M) ; \quad
		u \mapsto \frac{\sqrt{-1}}{2} \pqty{\tensor{u}{_{\ovxb}} \tensor{\theta}{^{\ovxb}}
			- \tensor{u}{_{\alpha}} \tensor{\theta}{^{\alpha}}} + \frac{1}{2n} (\Delta_{b} u) \theta
	\end{equation}
	is independent of the choice of $\theta$.
\end{lemma}

\begin{proof}
	From the transformation law of an admissible coframe \cref{eq:transformation-law-of-admissible-coframe}
	and the sub-Laplacian \cref{eq:transformation-law-of-sub-Laplacian},
	we obtain
	\begin{align}
		& \frac{\sqrt{-1}}{2} \pqty{ \tensor{u}{_{\ovxb}} \tensor{\hat{\theta}}{^{\ovxb}}
			- \tensor{u}{_{\alpha}} \tensor{\hat{\theta}}{^{\alpha}} }
			+ \frac{1}{2n} \pqty{ \widehat{\Delta}_{b} u } \hat{\theta} \\
		&= \frac{\sqrt{-1}}{2} \pqty{ \tensor{u}{_{\ovxb}} \tensor{\theta}{^{\ovxb}}
			- \tensor{u}{_{\alpha}} \tensor{\theta}{^{\alpha}} }
			+ \frac{1}{2} \pqty{ \tensor{u}{_{\ovxb}} \tensor{\Upsilon}{^{\ovxb}}
			+ \tensor{u}{_{\alpha}} \tensor{\Upsilon}{^{\alpha}} } \theta
			+ \frac{1}{2n} \pqty{ \Delta_{b} u - n \tensor{u}{_{\ovxb}} \tensor{\Upsilon}{^{\ovxb}}
			- n \tensor{u}{_{\alpha}} \tensor{\Upsilon}{^{\alpha}} } \theta \\
		&= \frac{\sqrt{-1}}{2} \pqty{ \tensor{u}{_{\ovxb}} \tensor{\theta}{^{\ovxb}}
			- \tensor{u}{_{\alpha}} \tensor{\theta}{^{\alpha}} } + \frac{1}{2n} (\Delta_{b} u) \theta,
	\end{align}
	which completes the proof.
\end{proof}

As in complex geometry,
CR pluriharmonic functions are smooth functions annihilated by $d d^{c}_{\CR}$.

\begin{lemma}
	For $u \in C^{\infty}(M)$,
	\begin{equation}
	\label{eq:expression-of-dd^c_CR}
		d d^{c}_{\CR} u
		= \sqrt{-1} \pqty{ \tensor{P}{_{\alpha}_{\ovxb}} u }
			\tensor{\theta}{^{\alpha}} \wedge \tensor{\theta}{^{\ovxb}}
			+ \pqty{ \tensor{P}{_{\alpha}} u } \theta \wedge \tensor{\theta}{^{\alpha}}
			+ \pqty{ \tensor{P}{_{\ovxb}} u } \theta \wedge \tensor{\theta}{^{\ovxb}},
	\end{equation}
	where
	\begin{gather}
		\tensor{P}{_{\alpha}_{\ovxb}} u
		= \tensor{u}{_{\alpha}_{\ovxb}} - \frac{1}{n} \tensor{u}{_{\gamma}^{\gamma}} \tensor{l}{_{\alpha}_{\ovxb}}
		= \tensor{u}{_{\ovxb}_{\alpha}}
			- \frac{1}{n} \tensor{u}{_{\ovxg}^{\ovxg}} \tensor{l}{_{\alpha}_{\ovxb}}, \\
		\tensor{P}{_{\alpha}} u
		= \frac{1}{n} \tensor{u}{_{\ovxg}^{\ovxg}_{\alpha}}
			+ \sqrt{- 1} \tensor{A}{_{\alpha}_{\gamma}} \tensor{u}{^{\gamma}},
		\qquad
		\tensor{P}{_{\ovxb}} u
		= \frac{1}{n} \tensor{u}{_{\gamma}^{\gamma}_{\ovxb}}
			- \sqrt{- 1} \tensor{A}{_{\ovxb}_{\ovxg}} \tensor{u}{^{\ovxg}}.
	\end{gather}
	In particular,
	u is a CR pluriharmonic function if and only if $d d^{c}_{\CR} u = 0$.
\end{lemma}

\begin{proof}
	We first show \cref{eq:expression-of-dd^c_CR}.
	From \cref{eq:definition-of-d^c_CR},
	it follows that
	\begin{align}
		d d^{c}_{\CR} u
		&= \frac{\sqrt{-1}}{2} \pqty{ \tensor{u}{_{\ovxb}_{\alpha}}
			+ \tensor{u}{_{\alpha}_{\ovxb}} + \frac{1}{n} (\Delta_{b} u) \tensor{l}{_{\alpha}_{\ovxb}} }
			\tensor{\theta}{^{\alpha}} \wedge \tensor{\theta}{^{\ovxb}} \\
		&\quad + \bqty{ - \frac{1}{2n} \tensor{(\Delta_{b} u)}{_{\alpha}}
			- \frac{\sqrt{-1}}{2} \tensor{u}{_{\alpha}_{0}}
			+ \frac{\sqrt{-1}}{2} \tensor{A}{_{\alpha}_{\gamma}} \tensor{u}{^{\gamma}} }
			\theta \wedge \tensor{\theta}{^{\alpha}} \\
		&\quad + \bqty{ - \frac{1}{2n} \tensor{(\Delta_{b} u)}{_{\ovxb}}
			+ \frac{\sqrt{-1}}{2} \tensor{u}{_{\ovxb}_{0}}
			- \frac{\sqrt{-1}}{2} \tensor{A}{_{\ovxb}_{\ovxg}} \tensor{u}{^{\ovxg}} }
			\theta \wedge \tensor{\theta}{^{\ovxb}}.
	\end{align}
	Combining this with \cref{eq:commutator-of-covariant-derivatives,eq:another-form-of-sub-Laplacian}
	yields that the $(1,1)$-part of $d d^{c}_{\CR} u$ is equal to $\sqrt{-1} \tensor{P}{_{\alpha}_{\ovxb}} u$.
	Hence it suffices to show
	\begin{equation}
		\tensor{P}{_{\alpha}} u
		= - \frac{1}{2n} \tensor{(\Delta_{b} u)}{_{\alpha}} - \frac{\sqrt{-1}}{2} \tensor{u}{_{\alpha}_{0}}
			+ \frac{\sqrt{-1}}{2} \tensor{A}{_{\alpha}_{\gamma}} \tensor{u}{^{\gamma}};
	\end{equation}
	the other part is the complex conjugate of this equality.
	By using \cref{eq:commutator-of-covariant-derivatives,eq:another-form-of-sub-Laplacian},
	we have
	\begin{align}
		&- \frac{1}{2n} \tensor{(\Delta_{b} u)}{_{\alpha}} - \frac{\sqrt{-1}}{2} \tensor{u}{_{\alpha}_{0}}
			+ \frac{\sqrt{-1}}{2} \tensor{A}{_{\alpha}_{\gamma}} \tensor{u}{^{\gamma}} \\
		&= \frac{1}{n} \tensor{u}{_{\ovxg}^{\ovxg}_{\alpha}}
			+ \frac{\sqrt{-1}}{2} \pqty{\tensor{u}{_{0}_{\alpha}} - \tensor{u}{_{\alpha}_{0}}}
			+ \frac{\sqrt{-1}}{2} \tensor{A}{_{\alpha}_{\gamma}} \tensor{u}{^{\gamma}} \\
		&= \frac{1}{n} \tensor{u}{_{\ovxg}^{\ovxg}_{\alpha}}
			+ \sqrt{-1} \tensor{A}{_{\alpha}_{\gamma}} \tensor{u}{^{\gamma}} \\
		&= \tensor{P}{_{\alpha}} u.
	\end{align}
	The latter statement is a consequence of \cite{Lee1988}*{Propositions 3.3 and 3.4}
	and the fact that
	\begin{equation}
		(\tensor{P}{_{\alpha}_{\ovxb}} u) \tensor{}{_{,}^{\ovxb}}
		= (n-1) \tensor{P}{_{\alpha}} u,
		\qquad
		(\tensor{P}{_{\alpha}_{\ovxb}} u) \tensor{}{_{,}^{\alpha}}
		= (n-1) \tensor{P}{_{\ovxb}} u;
	\end{equation}
	see \cite{Hirachi2014}*{Lemma 3.2}.
\end{proof}

A two-form $\mu$ on $M$ has \emph{trace-free $(1, 1)$-part} if
\begin{equation}
	\mu
	\equiv \sqrt{- 1} \tensor{\mu}{_{\alpha}_{\ovxb}} \theta^{\alpha} \wedge \theta^{\ovxb}
	\quad \text{modulo $\theta, \theta^{\alpha} \wedge \theta^{\gamma}, \theta^{\ovxb} \wedge \theta^{\ovxg}$}
\end{equation}
with $\tensor{\mu}{_{\alpha}^{\alpha}} = 0$.
Note that $d d^{c}_{\CR} u$ has trace-free $(1, 1)$-part,
which follows from \cref{eq:expression-of-dd^c_CR}.

\section{Strictly pseudoconvex domains}
\label{section:strictly-pseudoconvex-domains}

Let $\Omega$ be a relatively compact domain in an $(n+1)$-dimensional complex manifold $X$
with smooth boundary $M = \bdry \Omega$.
Denote by $\iota_{M}$ the inclusion $M \hookrightarrow X$.
There exists a smooth function $\rho$ on $X$ such that
\begin{equation}
	\Omega = \rho^{-1}((- \infty, 0)), \quad
	M = \rho^{-1}(0), \quad
	d \rho \neq 0 \ \quad \text{on} \ M;
\end{equation}
such a $\rho$ is called a \emph{defining function} of $\Omega$.
A domain $\Omega$ is said to be \emph{strictly pseudoconvex}
if we can take a defining function $\rho$ of $\Omega$ that is strictly plurisubharmonic near $M$.
The boundary $M$ is a closed strictly pseudoconvex real hypersurface
and $\iota_{M}^{\ast} d^{c} \rho$ is a contact form on $M$.
Conversely,
it is known that
any closed connected strictly pseudoconvex CR manifold of dimension at least five can be realized as
the boundary of a strictly pseudoconvex domain
in a complex projective manifold~\cites{Boutet_de_Monvel1975,Harvey-Lawson1975,Lempert1995}.

\subsection{Trace-free extension and $d^{c}_{\CR}$}

Assume that $M$ is realized as the boundary of a strictly pseudoconvex domain $\Omega$
in a complex manifold $X$ of complex dimension $n+1$.
Take a defining function $\rho$ of $\Omega$ with $\iota_{M}^{\ast} d^{c} \rho = \theta$.

\begin{lemma}
\label{lem:trace-free-extension}
	For each $u \in C^{\infty}(M)$,
	there exists a smooth extension $\widetilde{u}$ such that
	$\iota_{M}^{\ast} d d^{c} \wtu$ has trace-free $(1,1)$-part.
	Moreover,
	such a $\wtu$ is unique modulo $O(\rho^{2})$,
	and $\iota_{M}^{\ast} d^{c} \wtu$ coincides with $d^{c}_{\CR} u$.
\end{lemma}

\begin{proof}
	Take a smooth function $u^{\prime}$ on $X$ with $u^{\prime} |_{M} = u$.
	Then
	\begin{equation}
		\iota_{M}^{\ast} d^{c} u^{\prime}
		= \frac{\sqrt{-1}}{2} (\tensor{u}{_{\ovxb}} \tensor{\theta}{^{\ovxb}}
			- \tensor{u}{_{\alpha}} \tensor{\theta}{^{\alpha}})
			+ \lambda \theta
	\end{equation}
	for some $\lambda \in C^{\infty}(M)$.
	Hence the $(1,1)$-part of $\iota_{M}^{\ast} d d^{c} u^{\prime}$ is given by
	\begin{equation}
		\frac{\sqrt{-1}}{2}
		\pqty{ \tensor{u}{_{\ovxb}_{\alpha}} + \tensor{u}{_{\alpha}_{\ovxb}}
		+ 2 \lambda \tensor{l}{_{\alpha}_{\ovxb}} }
		\tensor{\theta}{^{\alpha}} \wedge \tensor{\theta}{^{\ovxb}}.
	\end{equation}
	On the other hand,
	the $(1,1)$-part of $\iota_{M}^{\ast} d d^{c} (\rho v)$ for $v \in C^{\infty} (X)$ coincides with
	\begin{equation}
		\sqrt{-1} v |_{M} \tensor{l}{_{\alpha}_{\ovxb}} \tensor{\theta}{^{\alpha}} \wedge \tensor{\theta}{^{\ovxb}}.
	\end{equation}
	If we choose $v$ so that $v |_{M} = (2n)^{-1} \Delta_{b} u - \lambda$,
	the $(1,1)$-part of $\iota_{M}^{\ast} d d^{c} (u + \rho v)$ is trace-free,
	which gives the existence of $\widetilde{u}$.
	It follows from the construction that
	$\iota_{M}^{\ast} d^{c} \widetilde{u} = d^{c}_{\CR} u$.
	The uniqueness of $\widetilde{u}$ modulo $O(\rho^{2})$
	is a consequence of the above computation of the $(1,1)$-part of
	$\iota_{M}^{\ast} d d^{c} (\rho v)$.
\end{proof}

\subsection{Asymptotically complex hyperbolic metrics}
\label{subsection:asymptotically-complex-hyperbolic-metrics}

Let $\Omega$ be a strictly pseudoconvex domain
in an $(n+1)$-dimensional complex manifold $X$ with $\bdry \Omega = M$.
Take a defining function $\rho$ of $\Omega$.
There exists a Hermitian metric $h$ on $K_{X}$ such that
\begin{equation}
	\omega_{+}
	= - d d^{c} \log (- \rho) + \sqrt{-1} (n+2)^{-1} \Theta_{h}
\end{equation}
defines a K\"{a}hler metric near the boundary
and satisfies
\begin{equation}
	\Ric_{\omega_{+}} + (n+2) \omega_{+}
	= d d^{c} O(\rho^{n+2});
\end{equation}
see~\cite{Hirachi2014}*{Section 2.2} for example.
To simplify notation,
we set $\Pi = \sqrt{-1} (n+2)^{-1} \Theta_{h}$.
Note that $\Pi$ is a closed real $(1, 1)$-form on $X$
and $-(n+2) \Pi$ is a representative of $2 \pi c_{1}(T^{1, 0} X)$.
We also note that
the pull-back of $\omega_{+}$ to the level set $\Set{\rho = - \epsilon}$
is equal to that of $\epsilon^{-1} d \vartheta + \Pi$,
where $\vartheta = d^{c} \rho$.

\subsection{CR $Q$-curvature}
\label{subsection:CR-Q-curvature}

Let $\Omega$ be a strictly pseudoconvex domain
in an $(n+1)$-dimensional complex manifold $X$ with $\bdry \Omega = M$.
Take a defining function $\rho$ of $\Omega$ with $\iota_{M}^{\ast} d^{c} \rho = \theta$.
Denote by $\Box_{+}$ the $\delb$-Laplacian with respect to $\omega_{+}$.
There exist $F, G \in C^{\infty}(\ovxco)$ such that
$F |_{M} = 0$
and
\begin{equation}
	U = \log (- \rho) - F - G (- \rho)^{n+1} \log (- \rho)
\end{equation} satisfies
\begin{equation}
	\Box_{+} U = n + 1 + O(\rho^{\infty}).
\end{equation}
Moreover,
the boundary value $G |_{M}$ of $G$ is given by
\begin{equation}
\label{eq:definition-of-CR-Q-curvature}
	G |_{M}
	= \frac{(- 1)^{n+1}}{n!(n+1)!} Q_{\theta},
\end{equation}
where $Q_{\theta}$ is the \emph{CR $Q$-curvature}.
This is a consequence of a characterization of the CR $Q$-curvature via $\Box_{+}$~\cite{Hirachi2014}*{Lemma 4.4}.

\section{Proof of main theorems}
\label{section:proof-of-main-theorems}

Let $(M, T^{1, 0} M)$ be as in \cref{thm:cohomological-expression-of-CR-Q-curvature}.
Without loss of generality,
we may assume that $M$ is connected.
There exists a strictly pseudoconvex domain $\Omega$ in an $(n+1)$-dimensional complex manifold $X$
with $\bdry \Omega = M$.
Fix a defining function $\rho$ of $\Omega$.
Let $\omega_{+} = - d d^{c} \log (- \rho) + \Pi$ and $U = \log (- \rho) - F - G (- \rho)^{n+1} \log (- \rho)$
be as in \cref{subsection:asymptotically-complex-hyperbolic-metrics} and \cref{subsection:CR-Q-curvature} respectively.
For smooth functions $f_{1}, f_{2}$ and an $(n,n)$-form $\Psi$ on $\Omega$,
we have
\begin{equation}
\label{eq:commutable-of-d-and-d^c}
	d f_{1} \wedge d^{c} f_{2} \wedge \Psi
	= \frac{\sqrt{-1}}{2} (\del f_{1} \wedge \delb f_{2} \wedge \Psi
		- \delb f_{1} \wedge \del f_{2} \wedge \Psi)
	= d f_{2} \wedge d^{c} f_{1} \wedge \Psi,
\end{equation}
which will be used repeatedly.

Let $\Upsilon$ be a CR pluriharmonic function on $M$.
Take its pluriharmonic extension to $\Omega$~\cite{Hirachi2014}*{Theorem A.1},
for which we use the same letter $\Upsilon$ by abuse of notation.
Note that $\iota_{M}^{\ast} d^{c} \Upsilon = d^{c}_{\CR} \Upsilon$,
which is a consequence of \cref{lem:trace-free-extension}.
It follows from \cref{eq:commutable-of-d-and-d^c} that
\begin{equation}
\label{eq:integrand-of-main-equality}
	d U \wedge d^{c} \Upsilon \wedge \omega_{+}^{n}
	= d \Upsilon \wedge d^{c} U \wedge \omega_{+}^{n}.
\end{equation}

We first show that $\lp \int_{\rho < - \epsilon}$
of the left hand side of \cref{eq:integrand-of-main-equality} exists.
On the one hand,
\begin{align}
	\int_{\rho < - \epsilon} d (G (- \rho)^{n + 1} \log (- \rho)) \wedge d^{c} \Upsilon \wedge \omega_{+}^{n}
	&= \int_{\rho = - \epsilon} (G \epsilon^{n + 1} \log \epsilon)
		d^{c} \Upsilon \wedge (\epsilon^{-1} d \vartheta + \Pi)^{n} \\
	&= O(1)
\end{align}
as $\epsilon \to + 0$.
On the other hand,
\begin{equation}
	d(\log (- \rho) - F) \wedge d^{c} \Upsilon \wedge \omega_{+}^{n}
	= \pqty{\frac{d \rho}{\rho} - d F} \wedge d^{c} \Upsilon \wedge \omega_{+}^{n}
\end{equation}
is $(- \rho)^{- n - 1}$ times a smooth form up to the boundary.
Hence
\begin{equation}
	\int_{\rho < - \epsilon} d(\log (- \rho) - F) \wedge d^{c} \Upsilon \wedge \omega_{+}^{n}
	=\sum_{m = 1}^{n} a_{m} \epsilon^{- m} + b \log \epsilon + O(1)
\end{equation}
as $\epsilon \to + 0$.
Therefore
$\lp \int_{\rho < - \epsilon} d U \wedge d^{c} \Upsilon \wedge \omega_{+}^{n}$ is well-defined.

We would like to compute $\lp \int_{\rho < - \epsilon}$ of both sides
of \cref{eq:integrand-of-main-equality}.
On the one hand,
\begin{align}
	\lp \int_{\rho < - \epsilon} d U \wedge d^{c} \Upsilon \wedge \omega_{+}^{n}
	&= \lp \int_{\rho < - \epsilon} d ( U d^{c} \Upsilon \wedge \omega_{+}^{n} ) \\
	&= \lp \int_{\rho = - \epsilon} U d^{c} \Upsilon \wedge \omega_{+}^{n} \\
	&= \lp \int_{\rho = - \epsilon} (\log \epsilon - F - G \epsilon^{n + 1} \log \epsilon)
		d^{c} \Upsilon \wedge (\epsilon^{-1} d \vartheta + \Pi)^{n} \\
	&= \sum_{k = 0}^{n} \binom{n}{k} \lp \epsilon^{- k} \log \epsilon
		\int_{\rho = - \epsilon} d^{c} \Upsilon \wedge (d \vartheta)^{k} \wedge \Pi^{n - k}.
\end{align}
If $k \geq 1$,
the integrand $d^{c} \Upsilon \wedge (d \vartheta)^{k} \wedge \Pi^{n - k}$
is $d$-exact on the closed manifold $\Set{ \rho = - \epsilon }$.
Hence Stokes' theorem yields
\begin{equation}
\label{eq:log-part-of-LHS}
	\lp \int_{\rho < - \epsilon} d U \wedge d^{c} \Upsilon \wedge \omega_{+}^{n}
	= \int_{M} d^{c} \Upsilon \wedge \Pi^{n}
	= \int_{M} d^{c}_{\CR} \Upsilon \wedge (\iota_{M}^{\ast} \Pi)^{n}.
\end{equation}

On the other hand,
\begin{align}
	&\lp \int_{\rho < - \epsilon} d \Upsilon \wedge d^{c} U \wedge \omega_{+}^{n} \\
	&= \lp \int_{\rho < - \epsilon} d (\Upsilon d^{c} U \wedge \omega_{+}^{n})
		- \lp \int_{\rho < - \epsilon} \Upsilon d d^{c} U \wedge \omega_{+}^{n} \\
	&= \lp \int_{\rho = - \epsilon} \Upsilon
		[- \epsilon^{- 1} \vartheta - d^{c} F - d^{c}(G (- \rho)^{n + 1} \log (- \rho))]
		\wedge (\epsilon^{-1} d \vartheta + \Pi)^{n} \\
	& \quad + \lp \int_{\rho < - \epsilon} (n + 1)^{- 1} \Upsilon (\Box_{+} U) \omega_{+}^{n + 1} \\
	&= \lp \int_{\rho = - \epsilon} \Upsilon
		[- (d^{c} G) \epsilon^{n + 1} \log \epsilon
		+ G ((n+1) \epsilon^{n} \log \epsilon + \epsilon^{n}) \vartheta]
		\wedge (\epsilon^{- 1} d \vartheta + \Pi)^{n} \\
	& \quad + \lp \int_{\rho < - \epsilon} \Upsilon \omega_{+}^{n + 1} \\
	&= \frac{(- 1)^{n + 1}}{(n!)^{2}} \int_{M} \Upsilon Q_{\theta} \theta \wedge (d \theta)^{n}
		+ \lp \int_{\rho < - \epsilon} \Upsilon \omega_{+}^{n + 1},
\label{eq:log-part-of-RHS}
\end{align}
where the last equality follows from \cref{eq:definition-of-CR-Q-curvature}.
Hence it suffices to compute the second term.
We divide $\Upsilon \omega_{+}^{n + 1}$ into two parts:
\begin{equation}
	\Upsilon \omega_{+}^{n + 1}
	= \Upsilon \Pi^{n + 1}
		+ \sum_{k = 1}^{n + 1} \binom{n + 1}{k} \Upsilon (- d d^{c} \log (- \rho))^{k} \wedge \Pi^{n + 1 - k}.
\end{equation}
First,
$\Upsilon \Pi^{n + 1}$ is smooth up to the boundary,
and so
\begin{equation}
	\lp \int_{\rho < - \epsilon}\Upsilon \Pi^{n + 1}
	= 0.
\end{equation}
Next,
for $1 \leq k \leq n + 1$,
\begin{align}
	&\lp \int_{\rho < - \epsilon} \Upsilon (- d d^{c} \log (- \rho))^{k} \wedge \Pi^{n + 1 - k} \\
	&= \lp \int_{\rho < - \epsilon} d [\Upsilon (- d^{c} \log (- \rho))
		\wedge (- d d^{c} \log (- \rho))^{k - 1} \wedge \Pi^{n + 1 - k}] \\
	& \quad + \lp \int_{\rho < - \epsilon} d \Upsilon \wedge d^{c} \log (- \rho)
		\wedge (- d d^{c} \log (- \rho))^{k - 1} \wedge \Pi^{n + 1 - k}.
\end{align}
From Stokes' theorem and \cref{eq:commutable-of-d-and-d^c},
we obtain
\begin{align}
	&\lp \int_{\rho < - \epsilon} \Upsilon (- d d^{c} \log (- \rho))^{k} \wedge \Pi^{n + 1 - k} \\
	&= \lp \int_{\rho = - \epsilon} \Upsilon (\epsilon^{- 1} \vartheta)
		\wedge (\epsilon^{-1} d \vartheta)^{k-1} \wedge \Pi^{n + 1 - k} \\
	& \quad + \lp \int_{\rho < - \epsilon} d \log (- \rho) \wedge d^{c} \Upsilon
		\wedge (- d d^{c} \log (- \rho))^{k-1} \wedge \Pi^{n + 1 - k}\\
	&= \lp \int_{\rho < - \epsilon} d [\log (- \rho) d^{c} \Upsilon
		\wedge (- d d^{c} \log (- \rho))^{k-1} \wedge \Pi^{n + 1 - k}] \\
	&= \lp \epsilon^{-k+1} \log \epsilon \int_{\rho = - \epsilon} d^{c} \Upsilon
		\wedge (d \vartheta)^{k - 1} \wedge \Pi^{n + 1 - k}.
\end{align}
If $k = 1$,
this yields that
\begin{equation}
	\lp \int_{\rho < - \epsilon} \Upsilon (- d d^{c} \log (- \rho)) \wedge \Pi^{n}
	= \int_{M} d^{c}_{\CR} \Upsilon \wedge (\iota_{M}^{\ast} \Pi)^{n}.
\end{equation}
If $k \geq 2$,
the integrand $d^{c} \Upsilon \wedge (d \vartheta)^{k - 1} \wedge \Pi^{n + 1 - k}$ is $d$-exact
on the closed manifold $\Set{\rho = - \epsilon}$.
Hence Stokes' theorem implies
\begin{equation}
	\lp \int_{\rho < - \epsilon} \Upsilon (- d d^{c} \log (- \rho))^{k} \wedge \Pi^{n + 1 - k}
	= 0.
\end{equation}
Thus we have
\begin{equation}
	\lp \int_{\rho < - \epsilon} \Upsilon \omega_{+}^{n + 1}
	= (n + 1) \int_{M} d^{c}_{\CR} \Upsilon \wedge (\iota_{M}^{\ast} \Pi)^{n}.
\end{equation}
Therefore
\cref{eq:log-part-of-RHS} yields
\begin{equation}
\label{eq:log-part-of-RHS-2}
	\begin{split}
		&\lp \int_{\rho < - \epsilon} d \Upsilon \wedge d^{c} U \wedge \omega_{+}^{n} \\
		&= \frac{(- 1)^{n + 1}}{(n!)^{2}} \int_{M} \Upsilon Q_{\theta} \theta \wedge (d \theta)^{n}
			+ (n + 1) \int_{M} d^{c}_{\CR} \Upsilon \wedge (\iota_{M}^{\ast} \Pi)^{n}.
	\end{split}
\end{equation}

\begin{proof}[Proof of \cref{thm:cohomological-expression-of-CR-Q-curvature}]
	We deduce from \cref{eq:integrand-of-main-equality,eq:log-part-of-LHS,eq:log-part-of-RHS-2} that
	\begin{equation}
		\int_{M} d^{c}_{\CR} \Upsilon \wedge (\iota_{M}^{\ast} \Pi)^{n}
		= \frac{(- 1)^{n + 1}}{(n!)^{2}} \int_{M} \Upsilon Q_{\theta} \theta \wedge (d \theta)^{n}
			+ (n + 1) \int_{M} d^{c}_{\CR} \Upsilon \wedge (\iota_{M}^{\ast} \Pi)^{n},
	\end{equation}
	or equivalently,
	\begin{equation}
		n \int_{M} d^{c}_{\CR} \Upsilon \wedge (\iota_{M}^{\ast} \Pi)^{n}
		= \frac{(- 1)^{n}}{(n!)^{2}} \int_{M} \Upsilon Q_{\theta} \theta \wedge (d \theta)^{n}.
	\end{equation}
	Since $- (n + 2) \iota_{M}^{\ast} \Pi$ is a representative of
	$2 \pi c_{1}(T^{1, 0} X |_{M}) = 2 \pi c_{1}(T^{1, 0} M)$,
	we have
	\begin{equation}
		\langle [d^{c}_{\CR} \Upsilon] \cup c_{1}(T^{1, 0} M)^{n}, [M] \rangle
		= \frac{(n + 2)^{n}}{n (n!)^{2} (2 \pi)^{n}}
			\int_{M} \Upsilon Q_{\theta} \theta \wedge (d \theta)^{n},
	\end{equation}
	which completes the proof.
\end{proof}

\begin{proof}[Proof of \cref{thm:orthogonality-of-CR-Q-curvature}]
	It follows from~\cite{Takeuchi2020-Chern}*{Theorem 1.1} that
	$c_{1}(T^{1, 0} M)^{n} = 0$ in $H^{2 n}(M, \bbR)$.
	Combining this fact with \cref{thm:cohomological-expression-of-CR-Q-curvature} yields
	\cref{thm:orthogonality-of-CR-Q-curvature}.
\end{proof}

\section*{Acknowledgements}

This work was motivated by a preprint version of~\cite{Case2021-Q-preprint}.
The author is grateful to Jeffrey Case for sharing it.
He also would like to thank Taiji Marugame and Yoshihiko Matsumoto for helpful comments.

\bibliography{my-reference,my-reference-preprint}

\end{document}